\documentclass[preprint]{elsarticle}
\makeatletter
\def\ps@pprintTitle{%
 \let\@oddhead\@empty
 \let\@evenhead\@empty
 \def\@oddfoot{\centerline{\thepage}}%
 \let\@evenfoot\@oddfoot}
\makeatother

\usepackage{amsmath}
\usepackage{amssymb}
\usepackage{tikz}
\usepackage{verbatim}
\usepackage{etoolbox}
\usepackage{IEEEtrantools}

\newtheorem{theorem}{Theorem}

\newenvironment{proof}{\par\addvspace{\baselineskip}\noindent\textit{Proof.}\;\relax}{\par\medskip}
\newenvironment{proofof}[1]{\par\addvspace{\baselineskip}\noindent\textit{Proof of #1.}\;\relax}{\par\medskip}
\newtheorem{lemma}{Lemma}

\newtheorem{con}[lemma]{Conjecture}

\newtheorem{cor}[lemma]{Corollary}

\begin{document}

\begin{frontmatter}

\title{Bounds on the number of edges of edge-minimal, edge-maximal and $l$-hypertrees}


\author[addr1]{P\'eter G. N. Szab\'o}
\ead{szape@cs.bme.hu}
\address[addr1]{Budapest University of Technology and Economics\\
Department of Computer Science and Information Theory\\
3-9., M\H{u}egyetem rkp., H-1111 Budapest, Hungary}

\begin{abstract}
In their paper, \textit{Bounds on the Number of Edges in Hypertrees}, G.~Y. ~Katona and P.~G.~N. ~Szab\'o introduced a new, natural definition of hypertrees in $k$-uniform hypergraphs and gave lower and upper bounds on the number of edges. They also defined edge-minimal, edge-maximal and $l$-hypertrees and proved an upper bound on the edge number of $l$-hypertrees.

In the present paper, we verify the asymptotic sharpness of the $\binom{n}{k-1}$ upper bound on the number of edges of $k$-uniform hypertrees given in the above mentioned paper. We also make an improvement on the upper bound of the edge number of $2$-hypertrees and give a general extension construction with its consequences.

We give lower and upper bounds on the maximal number of edges of $k$-uniform edge-minimal hypertrees and a lower bound on the number of edges of $k$-uniform edge-maximal hypertrees. In the former case, the sharp upper bound is conjectured to be asymptotically $\frac{1}{k-1}\binom{n}{2}$.
\end{abstract}

\end{frontmatter}


\section{Introduction}

The concept of chains was applied successfully in the generalisation of Ha\-mil\-ton\-i\-an-cycles to hypergraphs \cite{hamchains}. This definition seems to be useful for other purposes, for example, if one looks for an extension of a definition that involves paths to hypergraphs.
Based on this idea, a new concept for trees in $k$-uniform 
hypergraphs was introduced in \cite{Gyn}, and several different definitions were given for various types 
of hypertrees.

The authors then proved upper and lower bounds for the number of
edges in such hypertrees. First, we recall necessary definitions and summarize relevant theorems from \cite{Gyn}. In the present paper, the 
reader will find improvements of some of the earlier results, as well as 
new results in other areas of the topic.

Let $\mathcal{H}=(V,\mathcal{E})$ be a $k$-uniform hypergraph (with no multiple edges).
It is called a
\begin{itemize}
\item \textit{cycle} if there exists a cyclic
sequence $v_1,v_2,\ldots, v_l$ of its vertices such that every vertex
appears at least once (possibly more times) in it, and $\mathcal{E}$ consists of $l$ distinct edges of the form $\{v_{i}, v_{i+1},\ldots, v_{i+k-1}\}$, $1\leq i\leq l$;
\item \textit{semicycle} if there exists a
sequence $v_1,v_2,\ldots, v_l$ of its vertices such that every vertex
appears at least once (possibly more times) in it, $v_1=v_l$ and
$\mathcal{E}$ consists of $l-k+1$ distinct edges of the form $\{v_{i}, v_{i+1},\ldots, v_{i+k-1}\}$, $1\leq i\leq l-k+1$;
\item \textit{chain} if there exists a sequence
$v_1,v_2,\ldots, v_l$ of its vertices such that every vertex appears
at least once (possibly more times), $v_1\neq v_l$ and
$\mathcal{E}$ consists of $l-k+1$ distinct edges of the form $\{v_{i}, v_{i+1},\ldots, v_{i+k-1}\}$, $1\leq i\leq l-k+1$.
\end{itemize}

The \textit{length} of a cycle/semicycle/chain is the number of its edges. From the definition it follows that every semicycle has at least $3$ edges. A chain (semicycle) is \textit{non-self-intersecting} if every vertex appears exactly once in the defining sequence $v_1, v_2, \ldots, v_l$  (except for $v_1=v_l$).
It can be easily seen that if a $k$-uniform hypergraph
$\mathcal{H}$ contains a semicycle, then it contains also a
non-self-intersecting one,
and if $\mathcal{H}$ is semicycle-free, then every chain in it is non-self-intersecting (for detailed proofs see Section 2 in \cite{Gyn}).

As we mentioned earlier, chains play the most important role in defining
hypertrees because we intend to require a natural chain-connectedness
property.

A $k$-uniform hypergraph $\mathcal{H}$ is
\begin{itemize}
\item \textit{chain-connected} if every pair of its vertices is
connected by a chain, i.e., there exists a subhypergraph of it, which
is a chain and contains both vertices;
\item \textit{semicycle-free} if it contains no semicycle as a subhypergraph.
\end{itemize}

In \cite{Gyn}, The authors defined hypertrees by comparing equivalent definitions
of trees. Some of these definitions are not compatible with the concept of chain, while others
may be too general. One has to take into consideration that the original concept of cycle can be extended in two ways.

The $k$-uniform hypergraph $\mathcal{F}$ is a
\begin{itemize}
\item \textit{hypertree} if it is chain-connected and
semicycle-free;
\item \textit{$l$-hypertree} if it is a hypertree, and every chain in it has length at most $l$;
\item \textit{edge-minimal hypertree} if it is a hypertree, and deleting any edge
$e$, $\mathcal{F}\backslash\{e\}$ is not a hypertree any more
(i.e., chain-connectedness does not hold);
\item \textit{edge-maximal hypertree} if it is a hypertree, and adding any new
edge $e$, $\mathcal{F}\cup\{e\}$ is not a hypertree any more
(i.e., semicycle-freeness does not hold).
\end{itemize}

In this way, the edge maximal/minimal hypertrees are also common
hypertrees, but their edge-sets are extremely small/large. So, the last
two definitions describe the extreme cases among hypertrees.
One motivation to use the semicycle-free
property is that every chain is
non-self-intersecting in a hypertree, as we have mentioned previously. Without this property one must
face with substantially more complicated case-analyses.

Being connected by chains is not a transitive
property, thus it is not an equivalence relation. This is
a characteristic difference between hypertrees and common trees, which is responsible for most of the additional complexity.

We also remark that edge-minimality means that one can assign two vertices to every edge such that every chain connecting them contains the assigned edge, i.e., it must be in all of the minimal chains that connect these two vertices. One can rephrase that statement as follows: if for every two vertices of an edge-minimal hypertree we mark a chain connecting them, the marked chains together cover the whole edge-set of the hypertree.

Every $t$-$(n,k,1)$ block design is a hypertree (called \textit{$t$-geometric hypertree}) if $2\leq t\leq k-1$. This shows, that hypertrees can be considered as generalisations of $t$-block designs.

Finally, we summarise the already known results on the number of edges of hypertrees in the following theorems from \cite{Gyn}.
Let $\mathcal{F}=(V,\mathcal{E})$ be a $k$-uniform hypergraph, $n=|V|$ and $m=|\mathcal{E}|$.

\begin{theorem}[(Katona-Szab\'o \cite{Gyn})]\label{Tlower} If $\mathcal{F}$ is chain-connected and $n\geq
(k-1)^2$, then $m\geq n-(k-1)$, and this bound is tight.
\end{theorem}

The tightness of the above bound is obvious considering non-self intersecting chains. The condition $n\geq
(k-1)^2$ cannot be omitted if $k\geq 6$ because for such a $k$ there exists a $k$-uniform hypertree on $k+3$ vertices with $3$ edges.

\begin{theorem}[(Katona-Szab\'o \cite{Gyn})]\label{Tupper} If $\mathcal{F}$ is semicycle-free, then $m\leq \binom{n}{k-1}$, and this bound is asymptotically sharp for $k=3$.
\end{theorem}

\begin{theorem}[(Katona-Szab\'o \cite{Gyn})]\label{Tupper3} If $\mathcal{F}$ is an $l$-hypertree and $1\leq l\leq k$, then $m\leq
\frac{1}{k-l+1}\binom{n}{k-1}$.
This bound is asymptotically sharp in the case $l=2$, $k=3$.
\end{theorem}

In Section \ref{1}, we prove the asymptotic sharpness of Theorem \ref{Tupper} for every $k\geq 2$. Our recursive construction will be a $k$-hypertree, therefore it has some consequences for Theorem \ref{Tupper3}.
After that, we show some tools and results on 2-hypertrees in Section \ref{2}. We also prove a refined upper bound in case of $l=2$.
Finally, in sections \ref{4} and \ref{6}, we turn our attention to the edge number of edge-minimal and edge-maximal hypertrees, respectively. We give lower and upper bounds for the edge number and show a construction for a sequence of edge-minimal hypertrees with asymptotically as many edges as the conjectured sharp upper bound.


\section{Asymptotic sharpness of the upper bound of Theorem \ref{Tupper}}\label{1}

\begin{theorem}\label{Thuenk}
For every $k\geq 2$, there exists an infinite sequence of $k$-uniform hypertrees $\mathcal{H}_i^k$ with $n_i$ vertices and $e_i$ edges such that $\{n_i\}$ is strictly increasing and $e_i$ is asymptotically $\binom{n_i}{k-1}$.
\end{theorem}

Theorem \ref{Thuenk} implies that the bound of Theorem \ref{Tupper} is asymptotically sharp for all $k\geq 2$.
We call a 1-uniform hypergraph with vertex set $\{x_1,x_2,\ldots,x_{n-1}\}$, $n\geq 3$ and edge set $\{2\cdot\{x_1\},\{x_2\},\ldots,\{x_{n-1}\}\}$ (the multiplicity of the first edge is $2$) a 1-uniform semicycle of length $n$.

\begin{proofof}{Theorem \ref{Thuenk}}

The proof is divided into two lemmata. Lemma \ref{Lemma1} states that one can partition the set of $(k-1)$-subsets of $n$ into a few number of partition classes, such that no class contains a short semicycle. The second lemma constructs a suitable chain-connected hypergraph from that partition, which contains neither short nor long semicycles.

\begin{lemma}\label{Lemma1}
Let $m\geq 0$, $k\geq 2$ be positive integers and $n=2^m$ be such that $n\geq k-1$.
Then there exists a partition of the set $\binom{[n]}{k-1}$ to $F(n,k-1)\leq (\log n)^{k-2}=m^{k-2}$ classes such that every class covers $[n]$ and contains no semicycle of length at most $k$ (here and henceforth, $\log$ means $\log_2$).
\end{lemma}

\begin{proof}
We define the desired partition by a recursive construction.

Let $\mathcal{Q}_{n,k-1}=(Q_{n,k-1}^1,Q_{n,k-1}^2,\ldots,Q_{n,k-1}^{F(n,k-1)})$ denote the partition cor\-re\-spond\-ing to $\binom{[n]}{k-1}$, where $Q_{n,k-1}^i\subseteq \binom{[n]}{k-1}$ are the partition classes of $\mathcal{Q}_{n,k-1}$.

If $k=2$, then the partition has one class, $Q_{n,1}^1=\binom{[n]}{1}$. In this case $F(n,1)=1=(\log n)^{0}$, $Q_{n,1}^1$ covers $[n]$ and contains no semicycle of length at most 2 (a semicycle must have at least 3 edges even if the hypergraph is 1-uniform).

If $k\geq3$ and $n=k-1$, then the partition has also one class, $Q_{k-1,k-1}^1=\{[n]\}$. The statement of the theorem holds: $F(k-1,k-1)=1\leq (\log (k-1))^{k-2}$ because $k-1\geq 2$, $Q_{k-1,k-1}^1$ covers $[n]$ and contains no semicycle at all.

We define $\mathcal{Q}_{n,k-1}$ to be the empty set if $1\leq n<k-1$. In this case $F(n,k-1)=0\leq (\log n)^{k-2}$.

In any other case ($2<k\leq n=2^m$), assume that $\mathcal{Q}_{n',k'-1}$ is defined for all $k'<k$ and $n'$ or $k'=k$ and $n'<n$, where $n'$ is a power of 2.

We split $[n]$ into two parts, $V_1$ and $V_2$, each of size $n/2$. By induction, for every $1\leq\lambda\leq k-1$, there exist the appropriate partitions $\mathcal{Q}_{n/2,\lambda}^1=\{Q_{n/2,\lambda}^{1,1},\\\ldots,Q_{n/2,\lambda}^{1,F(n/2,\lambda)}\}$ and $\mathcal{Q}_{n/2,\lambda}^2=\{Q_{n/2,\lambda}^{2,1},\ldots,Q_{n/2,\lambda}^{2,F(n/2,\lambda)}\}$ of $V_1$ and $V_2$, respectively.

Let $\mathcal{Q}_{n,k-1}^*=\{Q_{n/2,k-1}^{1,1}\cup Q_{n/2,k-1}^{2,1},Q_{n/2,k-1}^{1,2}\cup Q_{n/2,k-1}^{2,2}, \ldots,Q_{n/2,k-1}^{1,F(n/2,k-1)}\cup Q_{n/2,k-1}^{2,F(n/2,k-1)}\}$ (if $n/2<k-1$, then $\mathcal{Q}_{n,k-1}^*=\emptyset$) and $\mathcal{Q}_{n,k-1}^\lambda=\{Q_{n/2,\lambda}^{1,i}\times Q_{n/2,k-1-\lambda}^{2,j}: 1\leq i \leq F(n/2,\lambda), 1\leq j \leq F(n/2,k-1-\lambda)\}$, where $A\times B$ denotes $\{a\cup b: a\in A, b\in B\}$ for convenience.

We define $\mathcal{Q}_{n,k-1}=\mathcal{Q}_{n,k-1}^*\cup\left(\bigcup_{\lambda=1}^{k-2} \mathcal{Q}_{n,k-1}^\lambda\right)$. We show that $\mathcal{Q}_{n,k-1}$ meets the conditions of the theorem.

First, we show that $\mathcal{Q}_{n,k-1}$ is a partition of $\binom{[n]}{k-1}$.

The classes of $\mathcal{Q}_{n,k-1}$ are disjoint:

If $Q_1$, $Q_2$ are two partition classes, $e\in Q_1\cap Q_2$, $|e\cap V_1|=\lambda$ and $0<\lambda<k-1$, then $Q_1,Q_2\in \mathcal{Q}_{n,k-1}^\lambda$ and there exist $i_1,i_2,j_1,j_2$ such that $Q_1=Q_{n/2,\lambda}^{1,i_1}\times Q_{n/2,k-1-\lambda}^{2,j_1}$, $Q_2=Q_{n/2,\lambda}^{1,i_2}\times Q_{n/2,k-1-\lambda}^{2,j_2}$. It means that $e\cap V_1\in Q_{n/2,\lambda}^{1,i_1}, Q_{n/2,\lambda}^{1,i_2}$, $e\cap V_2\in Q_{n/2,k-1-\lambda}^{2,j_1}, Q_{n/2,k-1-\lambda}^{2,j_2}$ hence $i_1=i_2$ and $j_1=j_2$ because $\mathcal{Q}_{n/2,\lambda}^1$ and $\mathcal{Q}_{n/2,k-1-\lambda}^2$ were partitions by induction. Thus $Q_1=Q_2$.

If $|e\cap V_1|=k-1$, then $Q_1,Q_2\in \mathcal{Q}_{n,k-1}^*$ and there exist $i,j$ such that $Q_1=Q_{n/2,k-1}^{1,i}\cup Q_{n/2,k-1}^{2,i}$, $Q_2=Q_{n/2,k-1}^{1,j}\cup Q_{n/2,k-1}^{2,j}$. It implies $e\in Q_{n/2,k-1}^{1,i}, Q_{n/2,k-1}^{1,j}$ and so $i=j$ because $\mathcal{Q}_{n/2,k-1}^1$ was a partition. Thus $Q_1=Q_2$. The case $|e\cap V_1|=0$ is similar.

Every edge $e$ is contained in a partition class.

Let $e_1=e\cap V_1$, $e_2=e\cap V_2$. If $|e_1|=\lambda$ and $0<\lambda<k-1$, then there exist classes $Q_{n/2,\lambda}^{1,i}\in\mathcal{Q}_{n/2,\lambda}^1$ and $Q_{n/2,k-1-\lambda}^{2,j}\in\mathcal{Q}_{n/2,k-1-\lambda}^2$ such that $e_1\in Q_{n/2,\lambda}^{1,i}$ and $e_2\in Q_{n/2,k-1-\lambda}^{2,j}$ because $\mathcal{Q}_{n/2,\lambda}^1$ and $\mathcal{Q}_{n/2,k-1-\lambda}^2$ were partitions. Therefore, $e\in Q_{n/2,\lambda}^{1,i}\times Q_{n/2,k-1-\lambda}^{2,j}\in\mathcal{Q}_{n,k-1}^\lambda$.

If $|e_1|=k-1$, then $e_1\in Q_{n/2,k-1}^{1,i}$ for some index $i$ because $\mathcal{Q}_{n/2,k-1}^1$ was a partition. Hence, $e_1\in Q_{n/2,k-1}^{1,i}\cup Q_{n/2,k-1}^{2,i}\in \mathcal{Q}_{n,k-1}^*$. The case $|e_1|=0$ is similar.

Let us continue with $F(n,k-1)=|\mathcal{Q}_{n,k-1}|\leq (\log n)^{k-2}$.
\begin{eqnarray*}
|\mathcal{Q}_{n,k-1}| & = & |\mathcal{Q}_{n,k-1}^*|+\sum_{\lambda=1}^{k-2}|\mathcal{Q}_{n,k-1}^\lambda|\\
& = & F(n/2,k-1)+\sum_{\lambda=1}^{k-2} F(n/2,\lambda)F(n/2,k-1-\lambda).
\end{eqnarray*}
By induction, this is at most
\begin{IEEEeqnarray*}{rCl}
\IEEEeqnarraymulticol{3}{l}{
(\log n/2)^{k-2}+\sum_{\lambda=1}^{k-2} (\log n/2)^{\lambda-1}(\log n/2)^{k-\lambda-2}
}\nonumber\\ \quad
& = & (\log n/2)^{k-2}+(k-2)(\log n/2)^{k-3}\\
& \leq & \sum_{i=0}^{k-2} \binom{k-2}{i}(\log n/2)^i\\
& \leq & (\log n/2+1)^{k-2}\\
& \leq & (\log n)^{k-2}.
\end{IEEEeqnarray*}

Every class of $\mathcal{Q}_{n/2,\lambda}^1$ and $\mathcal{Q}_{n/2,\lambda}^2$ covers $V_1$ and $V_2$ respectively, so every class $Q$ of $\mathcal{Q}_{n,k-1}$ covers $[n]$ either $Q\in\mathcal{Q}_{n,k-1}^*$ or $Q\in\mathcal{Q}_{n,k-1}^\lambda$.

Finally, we show that there is no class $Q$ of $\mathcal{Q}_{n,k-1}$ containing a semicycle of length at most $k$.

By the induction hypothesis, $\mathcal{Q}_{n,k-1}^*$ does not contain such semicycle.
Suppose that there is a short semicycle $C$ in some $Q\in \mathcal{Q}_{n,k-1}^\lambda$. We can assume that $C$ is non-self-intersecting (if there is a self-intersecting semicycle in a hypergraph, then there is a shorter non-self-inter\-sect\-ing one) and the first vertex of $C$ is in $V_1$. Project all of the edges of $C$ to $V_1$, delete the multiple edges, and denote the $\lambda$-uniform hypergraph obtained in this way by $C'$.

It's easy to see that $C'$ would be a $\lambda$-uniform, non-self-intersecting semicycle. 
Let $v_1,\ldots, v_l$ be the vertices of $C$ in the natural order (i.e., every $k-1$ consecutive vertices form an edge of $C$ and $v_1=v_l$) and let $u_i$ denote the $i$th vertex in this sequence that comes from $V_1$. Now, $V(C')=\{u_1, u_2,\ldots, u_{l'}\}$, where $u_1=v_1=v_l=u_{l'}$. It is enough to show that $\mathcal{E}(C')=\{\{u_i, u_{i+1}, \ldots, u_{i+\lambda-1}\} : 1\leq i\leq l'-\lambda+1\}$.
Obviously, $e\cap V_1$ is of the form $\{u_i, u_{i+1}, \ldots, u_{i+\lambda-1}\}$, for every edge $e$ of $C$. Every two consecutive edges of $C$ differ in only one vertex, hence this is true for the edges of $C'$. It proves our claim. $C'$ is non-self intersecting, because $C$ is non-self-intersecting.

The union of the first and last edges of $C$ covers all of its vertices because $C$ is a $(k-1)$-uniform semicycle of length at most $k$. This also holds for $C'$, and due to the non-self-intersecting property, the length of $C'$ is at most $\lambda+1$, which is a contradiction because $C'$ is a subhypergraph of $Q_{n/2,\lambda}^{1,i}$ for some $i$, and it could not contain a semicycle of length at most $\lambda+1$. $\square$
\end{proof}

\begin{lemma}\label{Lemma2}
Let $l=F(n,k-1)$, $n,m,k$ and $\mathcal{Q}_{n,k-1}=\{Q^1, Q^2, \ldots, Q^l\}$ be as in Lemma \ref{Lemma1}, and let $\mathcal{F}_{n,k}=(U_{n,k},\mathcal{D}_{n,k})$ be a hypertree, where $U_{n,k}=\{q_1, q_2, \ldots, q_l\}$ and $[n]$ are disjoint sets. Furthermore, let $V_{n,k}=[n]\cup U_{n,k}$ and $\mathcal{E}_{n,k}=\bigcup_{i=1}^l \{e\cup\{q_i\}: e\in Q^i\}\cup\mathcal{D}_{n,k}$. Then the hypergraph $\mathcal{H}_{n,k}=(V_{n,k}, \mathcal{E}_{n,k})$ is a $k$-uniform hypertree.
\end{lemma}

The set $U_{n,k}$ can be understood as a set of labels for the partition classes.
We construct $\mathcal{E}_{n,k}$ by labeling the edges of $\mathcal{Q}_{n,k-1}$ with a label from $U_{n,k}$, recording the class the edge belongs to.

\begin{proof}
\noindent $(1)$ chain-connectedness:

Let $u, v\in U_{n,k}$ be distinct vertices. Then they are connected by a chain because $\mathcal{F}_{n,k}$ is a hypertree.

If $u,v\in [n]$ and $k-1\geq 2$, then there exists a $(k-1)$-set $e\subseteq [n]$ containing them. This set is in $Q^i$ for some $i$, so $e\cup\{q_i\}$ is a chain of length 1 of $\mathcal{H}_{n,k}$ between $u$ and $v$.

If $u,v\in [n]$ and $k-1=1$, then $\{u\},\{v\}\in Q^1$, so $uq_1v$ is a chain between the two vertices.

In the case of $u\in [n]$ and $v=q_i$, there exists an edge $e\in Q^i$ such that $u\in e$ because $Q^i$ covers $[n]$. Hence, $u$ and $v$ are connected by the edge $e\cup \{q_i\}$.

\noindent $(2)$ Semicycle-freeness:

Assume to the contrary that $\mathcal{H}_{n,k}$ contains a semicycle denoted by $C$. An edge from $\mathcal{F}_{n,k}$ cannot be an edge of $C$ because otherwise $C$ would lie entirely in $\mathcal{F}_{n,k}$ in contradiction with the hypertree property. Now, every edge of $C$ contains exactly one vertex from $U_{n,k}$. Let $q_i$ denote such a vertex in the first edge. If the second edge contains $q_j$, where $j\neq i$, then the intersection with the first edge is of size at most $k-2$ (because these edges cannot be identical without $q_i$, otherwise they would correspond to the same partition class), which is a contradiction. By induction, this implies that every edge of $C$ contains $q_i$. This means that the length of $C$ is at most $k$. The $(k-1)$-uniform subhypergraph $C'$ obtained from $C$ by removing $q_i$ is clearly a semicycle of length at most $k$ in $Q^i$ ($q_i$ cannot be the initial vertex of $C$), which is impossible according to Lemma \ref{Lemma1}. $\square$
\end{proof}

Now, let $\mathcal{H}_i^k=\mathcal{H}_{2^i,k}=(V_{2^i,k},\mathcal{E}_{2^i,k})$, $n_i=|V_{2^i,k}|$ and $e_i=|\mathcal{E}_{2^i,k}|$. Then $n_i=2^i+F(2^i,k-1)$ and $e_i=\binom{2^i}{k-1}+|\mathcal{D}_{n,k}|$. By Theorem \ref{Tupper}, $F(2^i,k-1)\leq i^{k-2}$ and $e_i\geq \binom{2^i}{k-1}$ implies that $e_i$ is asymptotically $\binom{n_i}{k-1}$. $\square$
\end{proofof}

We remark that $\mathcal{H}_i^k$ is a $k$-hypertree: $\mathcal{H}_i^k$ contains no chain of length at least $k+1$, since the edges of any chain in $\mathcal{H}_i^k$ have a vertex in common. This proves the asymptotic sharpness of Theorem \ref{Tupper3} for $l=k$. It means that excluding long chains (of length at least $k+1$) has no effect on the asymptotic behaviour of the maximal edge number.

For $k=2$, $\mathcal{H}_i^k$ is a star, while for $k=3$, we get back $\mathcal{B}(\mathcal{F})$, the construction of Theorem $31$ from \cite{Gyn}, where $\mathcal{F}=\mathcal{F}_{n,3}$ of our Lemma \ref{Lemma2}.


\section{Results on 2-hypertrees}\label{2}

Theorem \ref{Tupper3} gives an upper bound for the number of edges of $l$-hypertrees which is conjectured to be sharp in asymptotic sense. In the following, we discuss 2-hypertrees and a corresponding equation called Star-equation, which is based on the star-decomposition of $2$-hypertrees.

The $k$-uniform hypergraph $\mathcal{S}_n$ of order $n$ is a \textit{(tight) star} if $n\geq k$, and all of the edges contain $k-1$ fixed vertices $u_1, u_2, \ldots, u_{k-1}$.
We call $\{u_1,u_2, \ldots, u_{k-1}\}$ the \textit{kernel} of the star.

It is easy to see that every star is an edge-minimal 2-hypertree with $n-k+1$ edges. 

\begin{lemma}[Star-decomposition]\label{Tsdisj}
If $\mathcal{H}=(V,\mathcal{E})$ is a $k$-uniform 2-hypertree, then any two distinct maximal stars of $\mathcal{H}$ are edge-disjoint.
\end{lemma}

\begin{proof} Let $\mathcal{C}_1=(V_1,\mathcal{E}_1)$ and $\mathcal{C}_2=(V_2,\mathcal{E}_2)$ be two distinct maximal stars of $\mathcal{H}$. It means that $\mathcal{C}_1$ (and similarly $\mathcal{C}_2$) is a subhypergraph of $\mathcal{H}$, which is a star, and any star which contains $\mathcal{C}_1$ as a subhypergraph is identical to $\mathcal{C}_1$.

Assume to the contrary that these two stars share an edge $e$.

There exist edges $e_1\in\mathcal{E}_1$ and $e_2\in\mathcal{E}_2$, both distinct from $e$, otherwise, one star would contain the other. By the definition of star, $|e\cap e_1|=|e\cap e_2|=k-1$. The kernels of $\mathcal{C}_1$ and $\mathcal{C}_2$ are $e\cap e_1$ and $e\cap e_2$, respectively. A maximal star is uniquely determined by its kernel, so $e\cap e_1\neq e\cap e_2$. However, in this case, $e_1$, $e$ and $e_2$ together form either a semicycle of length $3$ or a path of length $3$ (depending on whether $e_1\backslash e=e_2\backslash e$ or not), which is a contradiction. $\square$
\end{proof}

\begin{cor}
If $\mathcal{H}=(V,\mathcal{E})$ is a $k$-uniform 2-hypertree, then there is a unique decomposition of $\mathcal{E}$ into edge-disjoint maximal stars. 
\end{cor}
\begin{proof}
Every edge can be extended to a maximal star, and these stars are edge-disjoint due to Lemma \ref{Tsdisj}. $\square$
\end{proof}

Let $C_i$ and $l$ denote the number of stars with $i$ edges in the star-decomposition and the number of uncovered $(k-1)$-subsets of $V$, respectively. 

\begin{theorem}[(Star-equation)]\label{Tstar}
If $\mathcal{H}=(V,\mathcal{E})$ is a $k$-uniform 2-hypertree, then $$|\mathcal{E}|=\frac{1}{k-1}\binom{n}{k-1}-\frac{1}{k-1}\sum_{i=1}^{n-k+1} C_i-\frac{l}{k-1}.$$
\end{theorem}

\begin{proof}
First, we assign a kernel to every maximal star of $\mathcal{H}$. If a maximal star has at least 2 edges, then there is a natural choice of the kernel: it is the intersection of the edges. If a maximal star has only one edge, then we choose an arbitrary $(k-1)$-subset of the edge to be the kernel of it.

These star-kernels are pairwise distinct: if two maximal stars share the same kernel, then we can merge them to a larger star, which is impossible due to the maximality.

We count the $(k-1)$-subsets of $V$. Such a subset can be uncovered, a star kernel or covered, but not a star kernel.
The number of uncovered $(k-1)$-sets is $l$. The number of star kernels is equal to the number of maximal stars, which is $\sum_{i=1}^{n-k+1} C_i$.

Let us refer to the remaining $(k-1)$-sets as non-kernels, for simplicity.
Only one edge covers a non-kernel, otherwise, it would be the kernel of a maximal star. Every edge belongs to exactly one maximal star due to Lemma \ref{Tsdisj}. Hence, every non-kernel is a non-kernel of a uniquely determined maximal star. On the other hand, every non-kernel of a maximal star is a non-kernel of $\mathcal{H}$. So, the number of non-kernels is the sum of the number of non-kernels of maximal stars, which is $\sum_{i=1}^{n-k+1} (k-1)iC_i$.

Summing up the three cases, we have $$\binom{n}{k-1}=l+\sum_{i=1}^{n-k+1} C_i+(k-1)\sum_{i=1}^{n-k+1} iC_i.$$
On the other hand, $\sum_{i=1}^{n-k+1} iC_i=|\mathcal{E}|$ because every edge belongs to exactly one maximal star. $\square$
\end{proof}

The star-equation shows that if a sequence of 2-hypertrees reaches the upper bound of Theorem \ref{Tupper3}, then $l+\sum C_i$ must be $o(n^{k-1})$, or in other words we should cover almost all $(k-1)$-sets with a relatively few number of stars.
It is an interesting open philosophical question, whether we should use a block-design type construction with almost equally-sized stars or an imbalanced construction with some large stars as well as small ones filling the remaining gaps to reach the asymptotic upper bound. 

It turns out that one can refine the upper bound of Theorem \ref{Tupper3} in case of 2-hypertrees with a term of order $k-2$ by the help of the star-equation.

\begin{theorem}\label{Tupp2hyp}
If $\mathcal{H}=(V,\mathcal{E})$ is a $k$-uniform 2-hypertree, then $|\mathcal{E}|\leq \frac{1}{k-1}\binom{n}{k-1}-\frac{1}{(k-1)^3}\binom{n}{k-2}$.
\end{theorem}

\begin{proof}
We use the simple fact that $\sum_{i=1}^{n-k+1} C_i\geq \frac{1}{n-k+1} |\mathcal{E}|$, which follows from $|\mathcal{E}|=\sum_{i=1}^{n-k+1} iC_i$ and $\sum_{i=1}^{n-k+1} iC_i\leq (n-k+1)\sum_{i=1}^{n-k+1} C_i$.

Comparing it to the star-equation, we get
\begin{eqnarray*}
|\mathcal{E}| & \leq & \frac{1}{k-1}\binom{n}{k-1}-\frac{1}{(k-1)(n-k+1)} |\mathcal{E}|-\frac{1}{k-1} l\nonumber\\
 & \leq & \frac{1}{k-1}\binom{n}{k-1}-\frac{1}{(k-1)(n-k+1)} |\mathcal{E}|,\nonumber
\end{eqnarray*}
which implies, that
\begin{eqnarray*}
|\mathcal{E}| & \leq & \left((k-1)+\frac{1}{(n-k+1)}\right)^{-1} \binom{n}{k-1}.\nonumber\\
& \leq & \left(\frac{1}{k-1}-\frac{1}{(k-1)^2(n-k+1)+(k-1)}\right) \binom{n}{k-1}\nonumber\\
& \leq & \left(\frac{1}{k-1}-\frac{1}{(k-1)^2(n-k+2)}\right)\binom{n}{k-1}\nonumber\\
& \leq & \frac{1}{k-1}\binom{n}{k-1}-\frac{1}{(k-1)^3}\binom{n}{k-2}\nonumber
\end{eqnarray*}
$\square$
\end{proof}

%

Theorem \ref{Tupp2hyp} shows that in case of $l=2$ the bound of Theorem \ref{Tupper3} cannot be a tight, hence further improvements of the upper bound is needed. This is likely to be true for greater values of $l$ as well.

A $t$-$(n,k,\lambda)$ \textit{design}, in our terminology, is a $k$-uniform hypergraph on $n$ vertices, where every $t$-element subset of vertices is contained in exactly $\lambda$ edges. Though, no general way is known to decide whether a block design exists for a certain combination of parameters, the size of the design can be easily determined by its parameters: a $t$-$(n,k,\lambda)$ \textit{design} has exactly $\frac{\lambda}{\binom{k}{t}}\binom{n}{t}$ edges. An $S(k-1,k,n)$ Steiner system is a $(k-1)$-$(n,k,1)$ design.

It is easy to see that every $S(k-1,k,n)$ Steiner system is a $k$-uniform $1$-hypertree, if $k>2$: every $(k-1)$-subset of $V$ is contained in exactly $1$ edge, which ensures chain-connectedness ($k>2$ is needed here) and makes chains of length at least 2 impossible. These hypertrees have $\frac{1}{k}\binom{n}{k-1}$ edges by the above formula. Because every $1$-hypertree is in fact a $2$-hypertree, the existence of infinitely many $S(k-1,k,n)$ Steiner system for a fixed $k$ implies the existence of a sequence of $k$-uniform $2$-hypertrees with asymptotically $\frac{1}{k}\binom{n}{k-1}$ edges. Fortunately, this existence theorem was proved by Hanani for $k=4$ in 1960 \cite{quadruple}, and by Keevash in general in 2014 \cite{keevash}. The bound $\frac{1}{k}\binom{n}{k-1}$ is called the trivial lower bound for the edge number of $k$-uniform $2$-hypertrees and will be improved for $k=4$ by our forthcoming construction. Notice that it is already really close to the upper bound of Theorem \ref{Tupper3} obtained in \cite{Gyn}.
The fact that it counts as the ``trivial'' lower bound from our viewpoint shows the difficulty of the topic very well. If we want to know everything about hypertrees, we have to know everything about balanced incomplete block designs, which is known to be a challenging research area with long history.

Now, we show a general method to construct a $k$-uniform $2$-hypertree with high edge number from a given $S(k-1,k,n)$ Steiner system. We will apply it in Theorem \ref{Tordext} in order to improve the trivial lower bound in the $4$-uniform case.

Let $\mathcal{H}=(V,\mathcal{E})$ and $\mathcal{G}=(V,\mathcal{F})$ be a $k$-uniform and a $(k-1)$-uniform hypergraph, respectively, on the same vertex-set. We say that $\mathcal{H}$ is an \textit{extension} of $\mathcal{G}$ if every edge of $\mathcal{H}$ contains an edge of $\mathcal{G}$.

The edge $e\in\mathcal{E}$ is an \textit{extension} of $f\in\mathcal{F}$ (equivalently, $f$ is a \textit{kernel} of $e$) if $f\subset e$. We say that $f_1$ and $f_2$ are \textit{mutually extended} edges if \mbox{$|f_1\cap f_2|=k-3$}, and there exist $v_1\in f_1$ and  $v_2\in f_2$ such that $f_1\cup\{v_2\}\in\mathcal{E}$ and $f_2\cup\{v_1\}\in\mathcal{E}$.

\begin{lemma}[Extension]\label{Textension}
If $\mathcal{G}=(V,\mathcal{F})$ is an $S(k-2,k-1,n)$ Steiner system, and $\mathcal{H}=(V,\mathcal{E})$ is a chain-connected extension of $\mathcal{G}$ that does not contain mutually extended edges, then $\mathcal{H}$ is a $k$-uniform $2$-hypertree.
\end{lemma}

\begin{proof}
$\mathcal{H}$ is chain-connected by assumption. It is enough to show that $\mathcal{H}$ does not contain a semicycle of length 3 nor a chain of length 3 (a semicycle or a chain of length at least 4 contains a chain of length 3).

First, we note that every edge $e$ of $\mathcal{H}$ has a unique kernel. If $f_1, f_2\in\mathcal{F}$ are two distinct kernels of $e$, then $|f_1\cap f_2|=k-2$, which contradicts the definition of $\mathcal{F}$.

We claim that if two edges of $\mathcal{H}$ \---say, $e_1$ and $e_2$\--- intersect in a set of size $k-1$, then the intersection will be a kernel. There certainly exist $f_1, f_2\in\mathcal{F}$ such that $f_1\subset e_1$ and $f_2\subset e_2$ because $\mathcal{H}$ is an extension of $\mathcal{G}$. If $f_1, f_2\neq e_1\cap e_2$, then there exist vertices $u_1\in (e_1\cap e_2)\backslash f_1$ and $u_2\in (e_1\cap e_2)\backslash f_2$. Clearly, $f_1\cup \{u_1\}=e_1$ and $f_2\cup \{u_2\}=e_2$. Actually, $u_1\in f_2$, otherwise $f_2=e_2\backslash\{u_1\}$, and so $f_1\cap f_2=(e_1\cap e_2)\backslash\{u_1\}$ is a set of size $(k-2)$, which contradicts the properties of $\mathcal{G}$. Similarly, $u_2\in f_1$ and $f_1\cap f_2=(e_1\cap e_2)\backslash\{u_1, u_2\}$, thus $|f_1\cap f_2|=k-3$. In fact, we have just proved that $f_1$ and $f_2$ are mutually extended kernels, which is a contradiction.

Finally, if $P$ is a chain or a semicycle of length $3$ of $\mathcal{H}$ with edges $e_1, e_2$ and $e_3$, then $f_1=e_1\cap e_2$ and $f_2=e_2\cap e_3$ would be kernels intersecting each other in a set of size $k-2$, which is impossible. $\square$
\end{proof}

The simplest method to build an extension of a given $S(k-2,k-1,n)$ Steiner system is the ordered extension, where the vertex set is ordered linearly, and every edge is extended with vertices that are greater than the greatest vertex in the original edge. We must mention though that other, non-ordered extension methods may yield better constructions.

We indicate the ordering of the vertices by a permutation. We want to emphasize that the following construction works with every permutation, but it is not irrelevant which one have been chosen because it can strongly affect the number of edges.

\begin{lemma}[Ordered extension]\label{Cordext}
Let $\mathcal{G}=(V,\mathcal{F})$ be as in Lemma \ref{Textension}, $|V|=n$ and $(v_1, v_2, \ldots, v_n)$ be a permutation of the vertices such that $\{v_1,v_2, \ldots, v_{k-1}\}\in\mathcal{F}$. Furthermore let
\begin{itemize}
\item $\mathcal{F}_i=\{e\in\mathcal{F} : e\subseteq \{v_1, v_2,\ldots, v_i\},\, v_i\in e\}$, for $i=1,2,\ldots n$;
\item $\mathcal{E}_i=\{f\cup\{v_j\}:f\in\mathcal{F}_i, j>i\}$, for $i=k-1,k,\ldots, n-1$ and $\mathcal{E}=\bigcup_{i={k-1}}^{n-1}\mathcal{E}_i$.
\end{itemize}
Then $\mathcal{H}=(V,\mathcal{E})$ is a $k$-uniform $2$-hypertree.
\end{lemma}

\begin{proof}
We will check that the assumptions of Lemma \ref{Textension} are satisfied. We note that $\mathcal{F}_i=\emptyset$ if $i<k-1$, $\mathcal{F}_{k-1}=\{\{v_1,v_2,\ldots,v_{k-1}\}\}$ and $\mathcal{F}=\bigcup_{i={k-1}}^{n} \mathcal{F}_i$.

First, we show that $\mathcal{H}$ is an extension of $\mathcal{G}$.
Assume to the contrary that $f_1,f_2\in\mathcal{F}$ are mutually extended edges, $u_1\in f_1$, $u_2\in f_2$ such that $e_1=f_1\cup\{u_2\}\in\mathcal{E}$ and $e_2=f_2\cup\{u_1\}\in\mathcal{E}$. By the definition of $\mathcal{E}$, there exist indices $k-1\leq i, j\leq n-1$ such that $e_1\in\mathcal{E}_{i}$, $e_2\in\mathcal{E}_{j}$, and so there are $g_1\in\mathcal{F}_i$ and $g_2\in\mathcal{F}_j$ such that $g_1\subset e_1$ and $g_2\subset e_2$. But now $g_1=f_1$ and $g_2=f_2$ because every edge of $\mathcal{H}$ has a unique kernel. Without loss of generality, we can assume that $i\leq j$. If $u_1=v_l$ for some $1\leq l\leq n$, then $l>j$ comes from $e_2\in\mathcal{E}_j$ and $f_2=g_2$. However, $u_1\in f_1$ and $f_1=g_1\in\mathcal{F}_i$ implies $l\leq i$, which is a contradiction.

Now, we show chain-connectedness of $\mathcal{H}$. Let $\{v_1, v_2,\ldots v_{k-1}\}=f^*$ and $e_l=f^*\cup \{v_l\}$, for $l=k,k+1,\ldots, n$. Then $\mathcal{E}^*=\{ e_l: l\geq k\}$ is a subset of $\mathcal{E}$ and forms a $k$-uniform star on $V$, hence $\mathcal{H}$ is chain connected.

Finally, we can apply Lemma \ref{Textension} to finish the proof. $\square$
\end{proof}

Now, we show our the best construction for 4-uniform 2-hypertrees.

\begin{theorem}\label{Tordext}
There exists a sequence of 4-uniform 2-hypertrees with asymptotic edge number $\frac{2}{7}\binom{n}{3}$.
\end{theorem}

\begin{proof}
We use a well-known construction for Steiner triple systems. Let $n=2^m-1$, $V=\{v_1,\ldots,v_n\}$ and $V_j=\{v_i : 2^{j}\leq i\leq 2^{j+1}-1\}$, for $j=0, 1,\ldots, m-1$. Obviously, $V=\bigcup_{j=0}^{m-1} V_j$ and $|V_j|=n_j=2^j$.

We can find $2^{j}-1$ edge-disjoint perfect matchings on $V_j$ denoted by $M_j^1, M_j^2,\\\ldots, M_j^{n_j-1}$. Let $\mathcal{E}_{st}=\bigcup_{j=1}^{m-1}\bigcup_{i=1}^{n_j-1}\bigcup_{P\in M_j^i} \{P\cup\{v_i\}\}$ and $\mathcal{F}_{st}=(V,\mathcal{E}_{st})$. Now, $\mathcal{F}_{st}$ is a Steiner triple system for every fixed $m$.

We use Lemma \ref{Cordext} with $\mathcal{G}=\mathcal{F}_{st}$ and vertex sequence $\{v_1,v_2,\ldots,v_n\}$ to obtain a $4$-uniform $2$-hypertree $\mathcal{H}=(V,\mathcal{E})$.

$|\mathcal{E}|=\sum_{i=3}^{n-1} |\mathcal{E}_i|=
\sum_{i=3}^{n-1} |\mathcal{F}_i|(n-i)=\sum_{j=1}^{m-1} \sum_{i=2^j}^{2^{j+1}-1} (i-2^j)(2^m-1-i)\sim \sum_{j=0}^{m-1} \sum_{i=2^j}^{2^{j+1}} (i-2^j)(2^m-i) \sim
\sum_{j=0}^{m-1} \sum_{i=0}^{2^j} i(2^m-2^j-i) \sim \sum_{j=0}^{m-1} \sum_{i=0}^{2^j} (-i^2+i(2^m-2^j))\sim \sum_{j=0}^{m-1} \left( -\frac{1}{3} 2^{3j}+\frac{1}{2}2^m 2^{2j}-\frac{1}{2}2^{3j}\right) \sim
\sum_{j=0}^{m-1} \left(-\frac{5}{6} 2^{3j}+\frac{1}{2}2^m 2^{2j}\right) \sim -\frac{5}{6} \frac{8^m}{7}+\frac{1}{2}2^m \frac{4^m}{3} \sim
\frac{1}{21} 8^m\sim \frac{2}{7}\binom{n}{3}$. $\square$
\end{proof}

We note that the extension process can be used to reach the optimal asymptotic bound in the $3$-uniform case. For the detailed construction refer to Section \ref{6}.


\section{Edge-minimal hypertrees}\label{4}

In this section we turn our attention to the edge-minimal hypertrees. We concentrate on the upper bounds of the edge number and give an interesting construction for a sequence of edge-minimal hypertrees with asymptotic edge number $\frac{1}{k-1}\binom{n}{2}$. Based on that construction, we establish a conjecture about the asymptotic upper bound on the maximal number of edges. Finally, we show that any asymptotic upper bound of the form $\alpha\binom{n}{2}$ is indeed an upper bound for every $n$.

Before we continue, we remark that an edge-minimal hypertree has at least $n-k+1$ edges if $n\geq (k-1)^2$, and this bound is tight. This is a simple consequence of Theorem \ref{Tlower} and the fact that every non-self-intersecting chain is an edge-minimal hypertree.

\begin{theorem}\label{Tedgemin}
There exists an infinite sequence $\{\mathcal{H}_{m}\}$ of $k$-uniform edge-minimal hypertrees with $n_m$ vertices and $e_m$ edges such that $\{n_m\}$ is strictly increasing and $e_m$ is asymptotically $\frac{1}{k-1}\binom{n_m}{2}$.
\end{theorem}

\begin{proof}
Let $m$ be a positive integer divisible by $k-1$, $k\geq 3$ and $\mathcal{H}_m=(V_m,\mathcal{E}_m)$ be the $k$-uniform hypergraph defined as follows.

Let $V_m=\{v_{ij}: 1\leq i\leq l+1, 1\leq j\leq m\}$, where $l=\binom{m-1}{k-2}$, and let $n=(l+1)m$. $V_m$ can be understood as an $(l+1)\times m$ grid of vertices.

Since $m$ is divisible by $k-1$, by Baranyai's theorem \cite{baranyai}, there exists a 1-factorisation of the $(k-1)$-uniform complete hypergraph $\mathcal{K}_m^{(k-1)}$ on the set $[m]=\{1,\ldots,m\}$ into $l$ partitions. Let $B_1, B_2, \ldots, B_l$ denote these partitions. So, we know that for all $r\neq s$, $|B_r|=\frac{m}{k-1}$, $B_r\cap B_s=\emptyset$ and $\bigcup_{i=1}^l B_i =\binom{[m]}{k-1}$.

Let $\mathcal{H}_m$ be the hypergraph whose edges are all of the $k$-sets of the form $\{v_{ij}, v_{rs_1},v_{rs_2}, \ldots,v_{rs_{k-1}}\}$, where $1\leq j\leq m$, $1\leq i\leq l$, $i< r\leq l+1$ and $\{s_1, s_2, \ldots, s_{k-1}\} \in B_i$. We can imagine this such that for every $i$, $1\leq i\leq l$ a partition (namely $B_i$) is assigned to row $i$.
Projecting this partition onto every row with index greater than $i$, each $(k-1)$-set obtained in this way forms an edge with each vertex of row $i$.

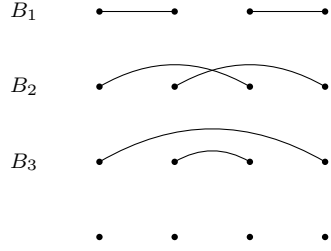
\begin{figure}[ht]
\centering
\begin{tikzpicture}
[pont/.style={circle, fill=black, inner sep=0.3mm}]

\foreach \x in {0,...,3}{
\foreach \y in {0,...,3}{
\node at (\x,\y) [pont] {};
}}
\draw (0,3)--(1,3);
\draw (2,3)--(3,3);
\draw[bend angle=30, bend left] (0,2) to (2,2);
\draw[bend angle=30, bend left] (1,2) to (3,2);
\draw[bend angle=30, bend left] (0,1) to (3,1);
\draw[bend angle=30, bend left] (1,1) to (2,1);

\draw (-1,3) node {\footnotesize$B_1$};
\draw (-1,2) node {\footnotesize$B_2$};
\draw (-1,1) node {\footnotesize$B_3$};

\end{tikzpicture}
\caption{The rows of the $3$-uniform hypergraph $\mathcal{H}_{4}$ with the matchings assigned to them.}
\end{figure}

First we show that $\mathcal{H}_m$ is an edge-minimal hypertree. Actually, a stronger result can be proven.

\begin{lemma}\label{Lmanyedge}
$\mathcal{H}_m$ is an edge-minimal 2-hypertree.
\end{lemma}

\begin{proof}
$\mathcal{H}_m$ is obviously a $k$-uniform hypergraph. It is enough to show that it is chain-connected, semicycle-free, edge-minimal and 2-hypertree.

\noindent $(1)$ Chain-connectedness:

Let $v_{ij}, v_{rs}\in V_m$ be vertices from different rows, where $i<r$. Since $B_i$ is a partition of $[m]$, there exist indices $s_2, \ldots, s_{k-1}\neq s$ such that $\{s, s_2, \ldots, s_{k-1}\}\in B_i$. Now, $\{v_{ij}, v_{rs}, v_{rs_2}, \ldots, v_{rs_{k-1}}\}\in\mathcal{E}_m$ by definition, so  $v_{ij}$ and $v_{rs}$ are connected by a chain of length 1.

We have to show chain connectedness of distinct vertices in the same row. Let $v_{ij_1}, v_{ij_2}$ be such vertices.

If $i<l+1$, then choosing an arbitrary $(k-1)$-set $\{s_1, s_2, \ldots, s_{k-1}\}$ from $B_i$, $\{v_{ij_1}, v_{i+1,s_1}, \ldots, v_{i+1,s_{k-1}}\}$, $\{v_{ij_2}, v_{i+1,s_1}, \ldots, v_{i+1,s_{k-1}}\}\in\mathcal{E}_m$, thus $v_{ij_1}$ and $v_{ij_2}$ are connected by a chain of length 2.

If $i=l+1$, then there exists a partition $B_r$ which has an edge $T=\{s_1, s_2, \ldots, s_{k-1}\}$ containing the pair $\{j_1,j_2\}$, due to $\bigcup_{q=1}^l B_q=\binom{[m]}{k-1}$. Hence, by definition, $\{v_{rs}, v_{is_1}, \ldots, v_{is_{k-1}}\}\in\mathcal{E}_m$ for an arbitrary index $s$, namely $v_{ij_1}$ and $v_{ij_2}$ are connected by a chain of length 1.

\noindent $(2)$ Semicycle-freeness and freeness of chains of length 3:

Let us notice the simple fact that for any two vertices $v_{ij}$ and $v_{rs}$ from different rows, there is exactly one edge that contains both of them. For example, if $i<r$, all suitable edges are of the form $\{v_{ij},v_{rs}, v_{rs_2},\ldots, v_{rs_{k-1}}\}$, where $T=\{s, s_2, \ldots, s_{k-1}\}\in B_i$, and such a $T$ is uniquely determined (exactly 1 partition class contains $s$).

From this observation follows that if $e=\{v_{ij}, v_{rs_1},v_{rs_2}, \ldots,v_{rs_{k-1}}\}$ is an edge of a chain of length $3$ in $\mathcal{H}_m$, then this is the last edge of it. If it does not hold, then there would be two different edges that intersect $e$ in distinct $(k-1)$-sets, but at least one of these edges has to contain two vertices of $e$ from different rows, (here we use $k\geq 3$). This is impossible since we have seen that only one edge can contain such pair. The same is true for semicycles instead of chains (however, keep in mind that in a semicycle the first and last vertices are identified).

So, every chain or semicycle has at most two edges (the first and last edges). Therefore, there is no chain of length at least 3 or semicycle of any length in $\mathcal{H}_m$ because even the shortest semicycle consists of 3 edges.

\noindent $(3)$ Edge-minimality:

Let us delete an arbitrary edge $e=\{v_{ij}, v_{rs_1}, \ldots, v_{rs_{k-1}}\}$ from $\mathcal{H}_m$. Note that $i<r$ by the definition of the edge-set. We show that the pairs $\{v_{ij}, v_{rs_h}\}$ become chain-disconnected, for all $1\leq h\leq k-1$.

It was shown above that at most one edge contains $v_{ij}$ and $v_{rs_h}$ together, and $e$ was that edge. Therefore they cannot be connected by a chain of length one. Part $(2)$ of the proof implies that only a chain of length 2 could connect them. Write down the row indices of its vertices in the natural sequence. This sequence is of the form $i,\underbrace{p,\ldots, p}_{k-1},r$, where $i,r<p$.

It means that there is a $(k-1)$-subset of the $p$th row, that forms two edges with two vertices from different rows. 

This is impossible because $B_i\cap B_r=\emptyset$. Thus, one cannot connect $v_{ij}$ and $v_{rs_h}$ without using the edge $e$. $\square$
\end{proof}

Now, since $|V_m|=n=m(l+1)$, the number of edges is
$|\mathcal{E}_m| = |\{\text{edges of row }\\ l\}|+|\{\text{edges of row } (l-1)\}| + \ldots+|\{\text{edges of the first row}\}|
=m\frac{m}{k-1}+2m\frac{m}{k-1}+\ldots+lm\frac{m}{k-1}
=\binom{l+1}{2}\frac{m^2}{k-1}\sim \frac{l^2m^2}{2(k-1)}
\sim \frac{n^2}{2(k-1)} \sim \frac{1}{k-1}\binom{n}{2}
$, which completes the proof of Theorem \ref{Tedgemin}. $\square$
\end{proof}

We remark that $|\mathcal{E}_m|=\frac{1}{k-1}\binom{n}{2}-(l+1) \left(\frac{1}{k-1} \binom{m}{2}\right)$, so the gap between the edge number and the asymptotic bound is, roughly speaking, $\frac{((k-2)!)^\frac{1}{k-1}}{2(k-1)} n^{1+\frac{1}{k-1}}$, and we miss exactly as many edges as we could have maximally placed inside the rows of $\mathcal{H}_m$. Based on this idea, one can slightly improve the construction of Theorem \ref{Tedgemin} by compressing the rows as much as possible. That, however, does not give as much improvement ($(\frac{1}{4}-\frac{1}{\sqrt{18}})n^\frac{3}{2}$ in $3$-uniform case) as complexity to the proof, hence we omit the details.

We also note that Theorem \ref{Tedgemin} shows that the bound of Theorem \ref{Tupper3} is asymptotically sharp for $l=2$ in 3-uniform case.

Next, we present our conjectured upper bound on the number of edges of edge-minimal hypertrees, however, we only prove a weaker result. These bounds are quite surprising because the order of magnitude does not depend on $k$.

\begin{con}\label{Smaxel}
For every $k$-uniform edge-minimal hypertree $\mathcal{F}=(V,\mathcal{E})$ on $n$ vertices, $|\mathcal{E}|\leq \frac{1}{k-1} \binom{n}{2}$ holds.
\end{con}

Although we conjecture that the number of edges in an edge-minimal hypertree is $O(n^2)$, we only prove the easier $O(n^3)$ upper bound.

\begin{theorem}\label{Tupperem}
For every $k$-uniform edge-minimal hypertree $\mathcal{F}=(V,\mathcal{E})$ on $n$ vertices, $|\mathcal{E}|\leq \frac{n(n-1)(n-k+1)}{2}$.
\end{theorem}

\begin{proof}
Let us count the edges of $\mathcal{F}$. For every pair of vertices $P$, there are some edges that really take part in connecting the pair, i.e., deleting such an edge makes the two vertices chain-disconnected. Let us denote the set of these edges by $S(P)$. Then $\bigcup_{P\in\binom{V}{2}} S(P)=\mathcal{E}$ due to the edge-minimality. The set $S(P)$ is certainly contained in every chain connecting the vertices of $P$, otherwise we could delete an edge of $S(P)$ without $P$ becoming chain-disconnected. So, $|S(P)|\leq n-(k-1)$ because every chain of $\mathcal{F}$ is of length at most $n-(k-1)$. Hence, $|\mathcal{E}|=|\bigcup_{P\in\binom{V}{2}} S(P)|\leq \sum_{P\in\binom{V}{2}} |S(P)|\leq (n-(k-1))\binom{n}{2}$. $\square$
\end{proof}

This bound guarantees that edge-minimal hypertrees cannot have $\Omega\left(n^{k-1}\right)$ edges. Moreover, the existing examples (such as the one in Lemma \ref{Lmanyedge}) suggest that the best candidate for the asymptotically sharp upper bound is $\frac{1}{k-1} \binom{n}{2}$.

An edge-minimal 3-uniform $1$-hypertree has at most $\frac{1}{3} \binom{n}{2}$ edges since it is a geometric hypertree. As we can seen in Theorem \ref{Tupper3}, the bound $\frac{1}{2}\binom{n}{2}$ of Conjecture \ref{Smaxel} is true for 3-uniform 2-hypertrees, and it is asymptotically sharp by Theorem \ref{Tedgemin}. The first open question is the case of $3$-uniform edge-minimal $3$-hypertrees. They cannot have more than $\binom{n}{2}$ edges by Theorem \ref{Tupper}, but we must actually take advantage of the edge-minimality in order to prove Conjecture \ref{Smaxel}.

Surprisingly, both the conjectured upper bound and the bound of Theorem \ref{Tupperem} are decreasing in $k$. Usually, if we let the uniformity parameter increase, the degree of freedom would grow with it, and there would be more structures satisfying the predefined conditions, hence one may expect an upper bound to increase as well. This either means that these bounds are not optimal or the number of edges follows a somewhat counterintuitive pattern.

An interesting generalisation of Theorem \ref{Tupperem} is an upper bound on the number of edges in $k$-uniform edge-minimal $l$-hypertrees.

\begin{theorem}
For every $k$-uniform edge-minimal $l$-hypertree $\mathcal{F}=(V,\mathcal{E})$ on $n$ vertices, $|\mathcal{E}|\leq l\frac{n(n-1)}{2}$.
\end{theorem}
\begin{proof}
Analogous to the proof of Theorem \ref{Tupperem}. For every pair $P$ of $V$, $|S(P)|\leq l$ because $\mathcal{F}$ is an $l$-hypertree. $\square$
\end{proof}

If $l$ is constant, we have $m=O(n^2)$ and this upper bound is far below the bound of Theorem \ref{Tupperem} or Theorem \ref{Tupper3}.

In the last part of this section we show that any asymptotic upper bound of the form $\alpha\binom{n}{2}$ is a real upper bound. We define the \textit{edge-ratio} of a $k$-uniform hypertree that has $n$ vertices and $m$ edges, to be $m/\binom{n}{2}$. In order to prove the above statement, we show that any $k$-uniform edge-minimal hypertree with edge-ratio $\alpha$ can be extended to an infinite sequence of $k$-uniform edge-minimal hypertrees with edge-ratio $\alpha$. This sequence of hypertrees can be obtained by the following gluing construction.

\begin{theorem}[(Gluing of hypertrees)]\label{Tgluing1}
Let $\mathcal{H}=(V,\mathcal{E})$ be an $S(2,l,n)$\\ Steiner system (i.e., every pair of points is contained in exactly one edge and $m=|\mathcal{E}|=\binom{n}{2}/\binom{l}{2}$).
Let $k\geq 3$, and suppose that for each $E_i\in \mathcal{E}$, a $k$-uniform hypergraph $\mathcal{F}_i=(E_i, \mathcal{E}_i)$ is given.
\begin{itemize}
\item If for all $i=1,2,\ldots, m$ $|\mathcal{E}_i| = \alpha\binom{l}{2}$, then $|\bigcup_{i=1}^m \mathcal{E}_i| = \alpha\binom{n}{2}$.
\item If for all $i=1,2,\ldots, m$ $\mathcal{F}_i$ is an edge-minimal hypertree, then \\$\mathcal{F}=(V,\bigcup_{i=1}^m \mathcal{E}_i)$ is also an edge-minimal hypertree.

In this case, the edge-minimal hypertree $\mathcal{F}$ is called the \textit{gluing} of the edge-minimal hypertrees $\{\mathcal{F}_i: E_i\in \mathcal{E}\}$ along the Steiner system $\mathcal{H}$.
\end{itemize}
\end{theorem}

\begin{proof}
Part 1 is obvious, since $|\mathcal{E}|=\binom{n}{2}/\binom{l}{2}$, $|\mathcal{E}_i| = \alpha\binom{l}{2}$ and $\mathcal{E}_i$ and $\mathcal{E}_j$ are disjoint if $i\neq j$, because $k\geq 3$.

The proof of part 2 is elementary, hence left to the reader. Notice that if two edge of $\bigcup_{i=1}^m \mathcal{E}_i$ intersect each other in $k-1$ vertices then they must belong to the same $\mathcal{E}_i$. It means that every chain or semicycle in $\mathcal{F}$ is a chain or semicycle in some $\mathcal{F}_i$. $\square$
\end{proof}

Let $\mathcal{F}$ be a $k$-uniform edge-minimal hypertree on $l$ vertices with edge-ratio $\alpha$. If $\mathcal{H}$ is an $S(2,l,n)$ Steiner system, let $\mathcal{H}(\mathcal{F})$ denote a gluing of $\binom{n}{2}/\binom{l}{2}$ identical copies of $\mathcal{F}$ along $\mathcal{H}$.
Based on the existence theorem of Lu and Ray-Chaudhuri from \cite{lu,chaudhuri} we know that there exists an infinite sequence $\mathcal{H}_1,\mathcal{H}_2, \ldots$ of $S(2,l,n)$ Steiner systems, thus $\mathcal{H}_1(\mathcal{F}),\mathcal{H}_2(\mathcal{F}), \ldots$ is a sequence of $k$-uniform edge-minimal hypertrees with edge-ratio $\alpha$.
It means that the supremum and the limit superior of the edge-ratios of $k$-uniform edge-minimal hypertrees must be equal.


\section{Edge-maximal Hypertrees in $3$-uniform Case}\label{6}

In this section, we show a construction of $3$-uniform edge-maximal hypertrees, and conjecture that the corresponding edge number is an asymptotic lower bound on the number of edges of $3$-uniform edge-maximal hypertrees.

\begin{theorem}\label{Temax}
For all $n>2$ even, there exists an edge-maximal 3-uniform hypertree $\mathcal{M}=(V,\mathcal{E})$ with $\frac{1}{2}\binom{n}{2}-\frac{1}{4}n$ edges.
\end{theorem}

\begin{proof}
First, let us define $\mathcal{M}$. Let $n>2$ be an even integer and $V=\{v_{ij}: 1\leq i\leq n/2,\, j=1,2\}$. The set of edges is $\mathcal{E}=\{\{v_{ij},v_{k1},v_{k2}\}:k<i\}$. If $v\in V$, then $\overline{v}$ denotes the pair of $v$, i.e., $$\overline{v}=
\left\{\begin{array}{ll} 
v_{k2}, & \text{if } v=v_{k1}\\
v_{k1}, & \text{if } v=v_{k2} 
\end{array}\right.$$

The second step is to show that $\mathcal{M}$ is an edge-maximal hypertree. The chain-connectedness and semicycle-freeness can be shown similarly as in the proof of Theorem \ref{Lmanyedge} (choose $B_i=\{1,2\}$ for all $i$), only edge-maximality remains.

Let $h$ be a new edge of $\mathcal{M}$. Then $h$ is of the form $\{v_{kl}, v_{i1}, v_{i2}\}$, for $i>k$ or $\{v_{ij}, v_{kl}, v_{rs}\}$, for $i>k>r$ (the other triples are in $\mathcal{E}$).
If $h=\{v_{kl}, v_{i1}, v_{i2}\}$, then the sequence $\overline{v}_{il} v_{il} v_{kl} \overline{v}_{kl} \overline{v}_{il}$ determines a semicycle in $\mathcal{M}$ because $\{v_{il}, v_{kl}, \overline{v}_{kl}\}$, $\{v_{kl}, \overline{v}_{kl}, \overline{v}_{il}\}\in \mathcal{E}$.

If $h=\{v_{ij}, v_{kl}, v_{rs}\}$, then the sequence $v_{ij} v_{kl} v_{rs} \overline{v}_{rs} v_{ij}$ determines a semicycle in $\mathcal{M}$ because $\{v_{kl}, v_{rs}, \overline{v}_{rs}\}$, $\{v_{rs}, \overline{v}_{rs}, v_{ij}\}\in \mathcal{E}$. Thus, $\mathcal{M}$ is edge-maximal.

The reader may easily verify that the number of edges is $\frac{n(n-2)}{4}=\frac{1}{2}\binom{n}{2}-\frac{1}{4}n$, which completes the proof. $\square$
\end{proof}

$\mathcal{M}$ is an ordered extension of the $1$-$(n,2,1)$ block design. It is actually a complete matching with edges $\{v_{i1},v_{i2}\}$, for $i=1,2,\ldots,n/2$, and if we apply Lemma \ref{Cordext} with the vertex-sequence $v_{11}, v_{12}, v_{21}, v_{22}, \ldots, v_{\frac{n}{2},1}, v_{\frac{n}{2},2}$, $\mathcal{M}$ is proved to be a 2-hypertree, and its edge number is asymptotically the bound we have stated in Theorem \ref{Tupper3} in the case of $l=2$ and $k=3$.

\begin{con}\label{Semax}
Every 3-uniform edge-maximal hypertree on $n$ vertices has at least $\frac{1}{2}\binom{n}{2}-O(n)$ edges.
\end{con}

If this conjecture is true, it would be asymptotically sharp due to Theorem \ref{Temax}. For greater uniformity parameters, the lower bound should probably be $\frac{1}{k-1}\binom{n}{k-1}$, but our evidences seems to be vague in this general case.

We close the section with an interesting lower bound on the number of edges of $k$-uniform edge-maximal hypertrees.

\begin{theorem}\label{Tturan}
If $\mathcal{F}=(V,\mathcal{E})$ is a $k$-uniform edge-maximal hypertree of order $n$, then $|\mathcal{E}|\geq \frac{1}{k(k-1)}\frac{n-k+1}{n-k+2}\binom{n}{k-1}$.
\end{theorem}
\begin{proof}
Let $T(n,k,r)$ be the usual hypergraph Tur\'an number, i.e., the minimal edge number of an $r$-uniform hypergraph that contains no independent set of size $k$.
Let $\mu(n)$ denote the minimal edge number of an edge-maximal hypertree of order $n$. For every $k$-set $s\subset V$, $s\notin\mathcal{E}$, there exists an edge $e\in \mathcal{E}$ such that $|s\cap e|=k-1$, otherwise $\mathcal{F}'=(V,\mathcal{E}\cup\{s\})$ would be a hypertree too, in contradiction with the edge-maximality. Let us form a $(k-1)$-uniform hypergraph $\mathcal{F}^{(k-1)}=(V,\mathcal{E}^{(k-1)})$ from $\mathcal{F}$ by exchanging every edge by its $(k-1)$-subsets. Then $k|\mathcal{E}|\geq|\mathcal{E}^{(k-1)}|$, and $\mathcal{F}^{(k-1)}$ contains no independent set of size $k$, thus $|\mathcal{E}^{(k-1)}|\geq T(n,k,k-1)$. Using de Caen's lower bound on Tur\'an numbers \cite{caen}, $T(n,k,k-1)\geq \frac{1}{k-1}\frac{n-k+1}{n-k+2}\binom{n}{k-1}$, so $$|\mathcal{E}|\geq \frac{1}{k} |\mathcal{E}^{(k-1)}|\geq \frac{1}{k(k-1)}\frac{n-k+1}{n-k+2}\binom{n}{k-1}.$$ $\square$
\end{proof}

Let us call a hypertree \textit{isolated} if it is both edge-minimal and edge-maximal.
An important consequence of Theorem \ref{Tturan} is that there are finitely many $k$-uniform isolated hypertrees if $k> 4$, since the edge number of an edge-minimal hypertree is $O(n^3)$ which is asymptotically less than the $\Omega(n^{k-1})$ bound stated in Theorem \ref{Tturan}, showing that there is a considerable gap between the edge numbers of edge-minimal and edge-maximal hypertrees. This fact has an interesting interpretation: if there is given a $k$-uniform hypertree with $k> 4$, one can add a new edge to it or delete an original edge of it without violating the hypertree property. An isolated hypertree would be an isolated point of the poset of hypertrees where the ordering is defined by the ``subhypergraph'' relation.
It is an open question that there are infinitely many isolated hypertrees in cases of $k=3$ and $k=4$. Of course, in case of $k=2$, every tree is isolated.


\section{Open problems}

There are many interesting open problems related to hypertrees. We have mentioned some obvious questions in this paper such that: ``What is the maximal number of edges in a $k$-uniform edge-minimal hypertree of order $n$?'' or ``What is the minimal number of edges in a $k$-uniform edge-maximal hypertree of order $n$?''.

We have stated the following two conjectures:

\begin{enumerate}
\item Every $k$-uniform edge-minimal hypertree on $n$ vertices has at most $\frac{1}{k-1} \binom{n}{2}$ edges.
\item Every 3-uniform edge-maximal hypertree on $n$ vertices has at least $\frac{1}{2}\binom{n}{2}-O(n)$ edges.
\end{enumerate}

It remained an open question, too, that the upper bound of the edge number of $k$--uniform $l$-hypertrees is asymptotically sharp or not for every fixed $k$ and $l$. We can also modify the definition of edge-minimal hypertrees slightly. Instead of edge-minimal hypertrees, it is interesting to study edge-minimal chain-connected hypergraphs. Similarly, we can study edge-maximal semicycle-free hypergraphs instead of edge-maximal hypertrees.

If we allow a chain to use an edge several times then some of our definitions and theorems would change significantly. One can, for example show forbidden substructures in edge-minimal chain-connected hypergraphs.

At the end of section \ref{6}, we introduced the isolated (simultaneously edge-minimal and edge-maximal) hypertrees. This is a small subclass of hypertrees, and has finite cardinality if $k> 4$. In case of $k=3$, the cardinality is not conjectured yet. However, if it is infinite, then our conjectures suggest that the asymptotic edge number of this family is $\frac{1}{2}\binom{n}{2}$.

Alexey Pokrovskiy and his research group at Freie Universität Berlin recently showed a great interest in these open questions, and showed the author some interesting constructions that might be valuable in future research of edge-minimal hypertrees.

\section*{Acknowledgement}
We would like to thank G. Y. Katona as well as the anonymous reviewers for their valuable advices.

The work reported in the paper has been developed in the framework of the
 project "Talent care and cultivation in the scientific workshops of BME"
project. This project is supported by the grant  T\'AMOP -
4.2.2.C-11/1/KONV-2012-0013.

The author is partially supported by the Hungarian National 
Research Fund (grant number K108947).


\end{document}